\documentclass[11pt,a4paper,oneside,reqno]{amsart}

\usepackage{latexsym}
\usepackage{amsmath}
\usepackage{amsthm}
\usepackage{amssymb}
\usepackage{amsfonts}
\usepackage{fancyhdr}
\usepackage{comment}
\usepackage{sidenotes}
\usepackage{mathrsfs}                           
\usepackage{color}

\usepackage{hyperref}

\usepackage{calc}
             {\begin{list}{\arabic{enumi}.}{\usecounter{enumi}%
              \setlength{\labelsep}{0.5em}%
              \settowidth{\labelwidth}{\arabic{enumi}.}%
              \setlength{\leftmargin}{\labelwidth+\labelsep}}}%
             {\end{list}}

\def \ba {\begin {eqnarray*} }
\def \ea {\end {eqnarray*} }
\def \beq {\begin {eqnarray}}
\def \eeq {\end {eqnarray}}

\def \p {\partial}

\newcommand{\mR}{\mathbb{R}}                    
\newcommand{\mC}{\mathbb{C}}                    
\newcommand{\R}{\mathbb{R}}                    
\newcommand{\abs}[1]{\lvert #1 \rvert}          
\newcommand{\norm}[1]{\lVert #1 \rVert}         

\newcommand{\eps}{\varepsilon}

\newcommand{\phg}{\text{phg}}

\newcommand{\supp}{\mathrm{supp}}

\theoremstyle{definition}
\newtheorem{thm}{Theorem}[section]
\newtheorem{prop}[thm]{Proposition}
\newtheorem{cor}[thm]{Corollary}
\newtheorem{lemma}[thm]{Lemma}

\newtheorem{question}{Question}[section]

\newtheorem{definition}{Definition}

\newtheorem*{remark}{Remark}

\numberwithin{equation}{section}

\title[The linearized Calder\'on problem]{The linearized Calder\'on problem in transversally anisotropic geometries}

\author{David Dos Santos Ferreira}
\address{Institut \'Elie Cartan, UMR 7502, Universit\'e de Lorraine, CNRS, INRIA, Campus des Aiguillettes B.P. 70239, F-54506 Vandoeuvre-l\`es-Nancy, France}
\email{ddsf@math.cnrs.fr}
\author{Yaroslav Kurylev}
\address{University College London}
\email{y.kurylev@ucl.ac.uk}
\author{Matti Lassas}
\address{Department of Mathematics and Statistics, University of Helsinki}
\email{matti.lassas@helsinki.fi}
\author{Tony Liimatainen}
\address{Department of Mathematics and Statistics, University of Helsinki}
\email{tony.liimatainen@helsinki.fi}
\author[M. Salo]{Mikko Salo}
\address{University of Jyvaskyla, Department of Mathematics and Statistics, PO Box 35, 40014 University of Jyvaskyla, Finland}
\email{mikko.j.salo@jyu.fi}

\date{\today}

\begin{document}

\begin{abstract}
In this article we study the linearized anisotropic Calder\'on problem. In a compact manifold with boundary, this problem amounts to showing that products of harmonic functions form a complete set. Assuming that the manifold is transversally anisotropic, we show that the boundary measurements determine an FBI type transform at certain points in the transversal manifold. This leads to recovery of transversal singularities in the linearized problem. The method requires a geometric condition on the transversal manifold related to pairs of intersecting geodesics, but it does not involve the geodesic X-ray transform which has limited earlier results on this problem.
\end{abstract}

\maketitle

\section{Introduction} \label{sec_intro}

The anisotropic Calder\'on problem amounts to determining a conductivity matrix in a domain from current and voltage measurements on the boundary, up to a change of coordinates fixing the boundary. It is well known that in dimensions $n \geq 3$, the problem can be formulated in terms of determining a Riemannian metric in a compact manifold with boundary, up to a boundary fixing diffeomorphism, from Cauchy data of harmonic functions. The problem has been solved for real-analytic metrics \cite{LeU, LaU, LTU} and on Einstein manifolds \cite{GS}, and two-dimensional versions of the problem are also well understood \cite{N, LaU, ALP}. A related conformal anisotropic Calder\'on problem has been solved recently for conformally real-analytic metrics in~\cite{LLS}.

The anisotropic Calder\'on problem for smooth manifolds in dimensions $n \geq 3$ remains an open problem. However, there has been considerable progress in the class of conformally transversally anisotropic (CTA) smooth manifolds. It was shown in \cite{DKSaU} that these manifolds may be amenable to the method of complex geometrical optics solutions which has been very successful in the Calder\'on problem in Euclidean domains. For such manifolds in a fixed conformal class, the Calder\'on problem has been solved in \cite{DKSaU} and \cite{DKLS} under additional restrictions on the transversal geometry (simplicity in \cite{DKSaU} and injectivity of the geodesic X-ray transform in \cite{DKLS}).

In this article we continue the study of \cite{DKSaU} and \cite{DKLS}, with the objective of reducing further the limitations on the transversal geometry. In particular, we wish to introduce an alternative to the geodesic X-ray transform which has restricted the previous results.

We will only consider a linearized version of the Calder\'on problem on transversally anisotropic manifolds. In this setting, we show that boundary measurements determine a FBI type transform at certain points in the transversal manifold. This will lead to recovery of singularities results in the transversal manifold. The difference from the previous works \cite{DKSaU}, \cite{DKLS} is that we use pairs of complex geometrical optics solutions that concentrate near different transversal geodesics, instead of concentrating near the same geodesic. This will provide both spatial and frequency localization in the transversal manifold, instead of providing integrals over geodesics. We can currently carry out this program only for the linearized problem, and only for pairs of nontangential geodesics that only intersect at one point.

Let us now state the problem in detail. It is known that the anisotropic Calder\'on problem in a fixed conformal class reduces to an inverse problem for the Schr\"odinger equation \cite{DKSaU}, and we will study the Schr\"odinger problem. Let $(M,g)$ be a compact oriented Riemannian manifold with smooth boundary, and let $q \in L^{\infty}(M)$. Assume that $0$ is not a Dirichlet eigenvalue of $-\Delta_g + q$ in $M$ where $\Delta_g$ is the (negative) Laplace-Beltrami operator, and consider the Dirichlet problem 
\[
(-\Delta_g + q)u = 0 \text{ in } M, \quad u|_{\partial M} = h
\]
where $h \in H^{1/2}(M)$. The Dirichlet-to-Neumann map is the operator 
\[
\Lambda_q: H^{1/2}(\partial M) \to H^{-1/2}(\partial M), \ \ h \mapsto \partial_{\nu} u|_{\partial M}
\]
where the normal derivative is defined in a weak sense. The (nonlinear) Calder\'on problem is to recover the potential $q$ from the knowledge of $\Lambda_q$, when $(M,g)$ is known. We consider instead the linearization of this problem at the zero potential $q = 0$. This is the statement that $f$ can be recovered from $(D\Lambda)_0(f)$, where $(D\Lambda)_0$ is the Fr\'echet derivative of $q \mapsto \Lambda_q$ at $q=0$. Uniqueness in the linearized problem reduces to the following question:

\begin{question} \label{question_main}
Suppose that $f \in L^{\infty}(M)$ satisfies 
\[
\int_M f u_1 u_2 \,dV = 0
\]
for all $u_j \in H^1(M)$ with $\Delta_g u_j = 0$ in $M$. Is it true that $f=0$?
\end{question}

The linearization argument leading to Question \ref{question_main} is essentially the same as in \cite{C}, where a positive answer was also given by using complex exponentials if $M$ is a domain in $\mR^n$ with Euclidean metric (\cite{C} considers the conductivity equation, but the Schr\"odinger case is similar). Question \ref{question_main} has a positive answer if $\dim(M)=2$ \cite{GT}, or if $\dim(M) \geq 3$, $(M,g)$ is a CTA manifold and additionally the geodesic X-ray transform on the transversal manifold is injective \cite{DKLS}. Both \cite{GT} and \cite{DKLS} actually solve the nonlinear Calder\'on problem, but the methods in these papers also settle Question \ref{question_main} under the stated conditions. Let us now define CTA manifolds.

\begin{definition}
Let $(M,g)$ be a compact oriented manifold with smooth boundary and with dimension $n = \dim(M) \geq 3$.
\begin{enumerate}
\item[(a)]
$(M,g)$ is \emph{transversally anisotropic} if $(M,g) \subset \subset (T,g)$ where $T = \mR \times M_0$ and $g = e \oplus g_0$ and where $(M_0,g_0)$ is any compact manifold of dimension $n-1$ and with smooth boundary.
\item[(b)]
$(M,g)$ is \emph{conformally transversally anisotropic} (CTA) if $(M,c^{-1}g)$ is transversally anisotropic for some smooth positive function $c$, so that $g = c(e \oplus g_0)$.
\end{enumerate}
In both cases we call $(M_0,g_0)$ the \emph{transversal manifold}.
\end{definition}

Let us also give definitions related to transversal geodesics.

\begin{definition}\label{transveral_geos}
Let $(M_0,g_0)$ be a compact oriented manifold with smooth boundary.
\begin{enumerate}
\item[(a)] 
A \emph{nontangential geodesic} is a geodesic $\gamma: [a,b] \to M_0$ such that $\gamma(a)$ and $\gamma(b)$ are on $\partial M$, $\gamma(t) \in M^{\mathrm{int}}$ when $a < t < b$, and $\dot{\gamma}(a)$ and $\dot{\gamma}(b)$ are nontangential vectors on $\partial M$.
\item[(b)]
A point $(z_0,\xi_0) \in T^* M_0$ is said to be \emph{generated by a pair of nontangential geodesics} if there are two nontangential unit speed geodesics $\gamma_1$ and $\gamma_2$ in $M_0$ with $\gamma_1(0) = \gamma_2(0) = z_0$ and 
$$
\dot{\gamma}_1(0) + \dot{\gamma}_2(0) = t_0 \xi_0
$$ 
for some $0<t_0 <2$. (Here $\xi_0$ is understood as an element of $T_{z_0}M$ by duality.)
\item[(c)] If a pair of nontangential geodesics intersect only at one point, then the pair of geodesics is called \emph{admissible}. If the geodesics in part (b) intersect only at $z_0$, we say that $(z_0,\xi_0)$ is \emph{generated by an admissible pair of geodesics}.
\end{enumerate}
\end{definition}

If $(M,g)$ is a CTA manifold, we write $x = (x_1,x')$ for coordinates on $M$ where $x_1$ is the Euclidean coordinate and $x'$ are coordinates on $M_0$. If $f$ is a function on $M$, extended by zero to $\mR \times M_0$, we write 
\[
\hat{f}(\lambda,x') = \int_{-\infty}^{\infty} e^{-i\lambda x_1} f(x_1,x') \,dx_1
\]
for the Fourier transform with respect to $x_1$.

Our first main theorem states that if $f$ is orthogonal to products of harmonic functions, then $\hat{f}(\lambda,\,\cdot\,)$ must be smooth at any $(z_0, \xi_0)\in T^*M_0$, which has a neighborhood in $T^*M_0$ where every point is generated by an admissible pair of geodesics. 

\begin{thm} \label{thm_main1}
Let $(M,g)$ be a transversally anisotropic manifold, and suppose that $f \in L^{\infty}(M)$ satisfies 
\[
\int_M f u_1 u_2 \,dV = 0
\]
for all $u_j \in H^1(M)$ with $\Delta_g u_j = 0$ in $M$. Then for any $\lambda \in \mC$ one has 
\[
(z_0, \xi_0) \notin WF(\hat{f}(\lambda, \,\cdot\,))
\]
whenever $(z_0, \xi_0) \in T^*M_0$ has a neighborhood where every point is generated by an admissible pair of geodesics.
\end{thm}

\begin{remark}
It is likely that an analogous statement holds for the analytic wave front set if the manifold $(M_0,g_0)$ is real-analytic. This would follow from a construction of exponentially accurate quasimodes concentrating near a nontangential geodesic on a real-analytic manifold; since we could not find an exact statement in the literature, we have left this case to a forthcoming work. It is likely that this would also lead to a full solution of the linearized problem if any point $(z_0, \xi_0) \in T^* M_0$ has a neighborhood where every point is generated by an admissible pair of geodesics. Note that even if $(M_0,g_0)$ is real-analytic, this problem would correspond to deforming a real-analytic potential by a smooth perturbation, and so the result would not reduce to known results in the real-analytic case.
\end{remark}

It was required in Theorem~\ref{thm_main1} that $(z_0,\xi_0)$ has a neighborhood in $T^*M_0$ generated by admissible geodesic pairs. The reason for this requirement is that our study of the wave front set requires small perturbations of $(z_0,\xi_0)$. We give a geometric criterion for this requirement to hold true. The condition we give is closely related to a geometric regularity condition of~\cite{StU_geodesic_nonsimple}. Their condition essentially requires that any covector $(x,\eta)$ has an orthogonal covector that defines a geodesic $\gamma$ without conjugate points.

In our definition we require that the point $x$ from where we issue the corresponding geodesic $\gamma$ has no points conjugate to it. We additionally require that $\gamma$ does not self-intersect, which is allowed in~\cite{StU_geodesic_nonsimple}. Thus the definitions are close, but not completely the same. This is why we call our condition \emph{strict Stefanov-Uhlmann regularity}.

\begin{definition}[Strict Stefanov-Uhlmann regularity]\label{SU_compl}
Let $(M,g)$ be a manifold with boundary. Then $(M,g)$ satisfies the strict Stefanov-Uhlmann regularity condition at $\eta\in S^*_xM$ if there exists $\xi\in S_x^*M$, with
$g(\xi,\eta)=0$, such that the geodesic $\gamma=\gamma_{\xi}$ corresponding to $\xi$ satisfies:
\begin{enumerate}
 \item[(a)] The geodesic $\gamma$ is nontangential and defined on the interval $[-t_{\mathrm{in}},t_{\mathrm{out}}
 ]$. 
 \item[(b)] $\gamma$ contains no points conjugate to $x$.
 \item[(c)] The geodesic $\gamma$ does not self-intersect for any time $t\in [-t_{\mathrm{in}},t_{\mathrm{out}}]$.
\end{enumerate}
%
\end{definition}

If $(M_0,g_0)$ satisfies the conditions of the definition at $(z_0,\xi_0)\in T^*M_0$, we will show in Section~\ref{sec_Stefanov_Uhlmann} that $(z_0,\xi_0)$ has a neighborhood generated by admissible geodesic pairs. Together with Theorem~\ref{thm_main1}, this gives:
\begin{cor}
 In the setting of Theorem~\ref{thm_main1}, assume that $(z_0,\xi_0)\in T^*M_0$ satisfies the strict Stefanov-Uhlmann regularity condition. Then $(z_0, \xi_0) \notin WF(\hat{f}(\lambda, \,\cdot\,))$ for any $\lambda \in \mC$.
\end{cor}

We now sketch the argument for proving Theorem~\ref{thm_main1}. We consider harmonic functions in $(M,g)$ of the form 
\begin{align*}
u_1 &= e^{-s_1 x_1}(v_{s_1}(x') + r_1), \\
u_2 &= e^{s_2 x_1}(w_{s_2}(x') + r_2),
\end{align*}
where $s_j = \tau+i\lambda_j$ is a complex frequency, and $v_{s_1}$ and $w_{s_2}$ are quasimodes in $(M_0,g_0)$ such that as $\tau \to \infty$ 
\begin{gather*}
\norm{(-\Delta_{g_0}-s_1^2)v_{s_1}}_{L^2(M_0)} = \norm{(-\Delta_{g_0}-s_2^2)w_{s_2}}_{L^2(M_0)} = O(\tau^{-\infty}),  \\ \norm{v_{s_1}}_{L^2(M_0)} =  \norm{w_{s_2}}_{L^2(M_0)} = 1.
\end{gather*}
It is well known (see for instance \cite{DKLS}) that one can construct such quasimodes that concentrate near nontangential geodesics. If the above conditions are satisfied, the Carleman estimate in \cite{DKSaU} allows us to find correction terms $r_1$ and $r_2$ so that $u_1$ and $u_2$ are indeed harmonic and the correction terms satisfy $\norm{r_j}_{L^2(M)} = O(\tau^{-\infty})$. Inserting these functions $u_1$ and $u_2$ in the identity $\int f u_1 u_2 \,dV = 0$, we obtain 
\begin{equation}\label{integral_O_infty}
\int_{M_0} \hat{f}(\lambda,\,\cdot\,) v_{s_1} w_{s_2} \,dV_{g_0} = O(\tau^{-\infty})
\end{equation}
as $\tau \to \infty$, where 
$$
\lambda = \lambda_1-\lambda_2.
$$
We remark that the Carleman estimate in~\cite{DKSaU} is indeed needed to construct such solutions $u_1$ and $u_2$ even in the linearized problem.

In the works \cite{DKSaU} and \cite{DKLS} concerning the nonlinear problem, one takes the limit as $\tau \to \infty$ and this essentially forces one to use quasimodes $v_{s_1}$ and $w_{s_2}$ that concentrate near the \emph{same geodesic}. The reason is that if the quasimodes concentrate near different geodesics only intersecting at $z_0$, then $v_{s_1} w_{s_2} = e^{i\tau\psi} a$ where $a$ is supported near $z_0$ and $\psi$ has nonvanishing gradient near $z_0$. The resulting integral decays rapidly in $\tau$ by non-stationary phase, and one loses information about $f$ in the limit.

In this paper, we will 
\begin{itemize}
\item 
use $v_{s_1}$ and $w_{s_2}$ that concentrate near \emph{different geodesics}, and 
\item 
consider \emph{all values} $\tau_0 < \tau < \infty$ instead of taking the limit $\tau \to \infty$.
\end{itemize}
Using all values of $\tau$ would be challenging in the nonlinear problem, since one would need asymptotic expansions of quasimodes up to high order and the unknown potential $q$ would appear in the expansions. However, in the linearized problem it is enough to construct the quasimodes for $q=0$ and one can consider expansions to arbitrarily high order.

Now, if $v_{s_1}$ and $w_{s_2}$ concentrate near nontangential geodesics $\gamma_1$ and $\gamma_2$ that only intersect at $z_0$ when $t=0$, then the product $v_{s_1} w_{s_2}$ is supported in a small neighborhood of $z_0$ and $v_{s_1} w_{s_2} = e^{i\tau\psi} a$ where $\nabla \psi(z_0) = \dot{\gamma}_1(0) + \dot{\gamma}_2(0)$. This results in an FBI type transform that can be used to characterize the wave front set of $\hat{f}(\lambda,\,\cdot\,)$ at any such $(z_0,\xi_0)$ where $\xi_0 = \dot{\gamma}_1(0) + \dot{\gamma}_2(0)$.  

One could also consider the case where $v_{s_1}$ and $w_{s_2}$ concentrate near nontangential geodesics that intersect several times. In this case, the product $v_{s_1} w_{s_2}$ is supported in the union of small neighborhoods of the intersection points, and each intersection point produces a contribution in the integral. Thus there will be several terms whose sum is $O(\tau^{-\infty})$, but it is not clear to us at the moment how to separate the contributions from the different intersection points.

We remark that the results are given on transversally anisotropic manifolds instead of CTA manifolds. The reason is that the standard reduction from $\Delta_{cg}$ to $\Delta_g$ as in \cite{DKSaU} produces a potential, and if $c$ depends on $x_1$ then the potential would also depend on $x_1$. This is not compatible with the separation of variables argument here. However, our method applies with small modifications if the conformal factor only depends on $x'$, and similarly one could include a potential only depending on $x'$.

Let us conclude with some further references on the linearized Calder\'on problem. It seems to us that in many cases where uniqueness is known in the linearized problem, one also knows uniqueness in the corresponding nonlinear problem. We refer to the survey \cite{Uhlmann2014} for references on the Calder\'on problem in general. Concerning the linearized problem, \cite{DKSjU09} solves the Calder\'on problem with partial data linearized at $q=0$, in a Euclidean domain $\Omega \subset \mR^n$ with measurements on a fixed subset of $\partial \Omega$. The argument involves analytic microlocal analysis and Kashiwara's watermelon theorem. This result has been extended in \cite{SjU} to the Calder\'on problem linearized at a real-analytic potential with measurements on a real-analytic part of the boundary. We remark that corresponding results are open for the nonlinear Calder\'on problem if $n \geq 3$ (see the survey \cite{KSa_survey}). Linearized Calder\'on type problems on Riemannian manifolds are discussed in~\cite{Sh09}.

This paper is organized as follows. Section \ref{sec_intro} is the introduction. In Section \ref{sec_simple} we prove recovery of singularities results when the transversal manifold is simple. Much stronger results are known in this case \cite{DKSaU}, but the discussion here illustrates our method in an easy setting. In Section~\ref{sec_Stefanov_Uhlmann} we show that the strict Stefanov-Uhlmann regularity condition is sufficient for finding enough admissible geodesic pairs. Section \ref{sec_quasimodes} gives a construction and parametrization of Gaussian beams. The construction is well known, but we give the argument in detail since this will be needed later. Finally, Section \ref{sec_fbi} shows that one can recover FBI type transforms (see~\cite{Folland},~\cite{WZ}) from our data, and proves Theorem \ref{thm_main1}.

\subsection*{Acknowledgements}
The authors would like to thank Jared Wunsch for clarifying details of the work~\cite{WZ}. M.L., T.L.\ and M.S.\ were supported by the Academy of Finland (Centre of Excellence in Inverse Problems Research, grant numbers 284715 and 309963). T.L.\ and M.S.\ were also partly supported by an ERC Starting Grant (grant number 307023). 
D.DSF.\ was partially supported by l'Agence Nationale pour la Recherche under grant ANR-13-JS01-0006 \textit{iproblems}.

\section{Simple transversal manifolds} \label{sec_simple}
As a motivation, we first consider the case where the transversal manifold is simple. In this section, we will construct quasimodes $v_s$ and $w_s$ that concentrate near different geodesics on simple manifolds. In the simple case we already know the full result based on injectivity of the ray transform~\cite{DKSaU}, but it will be useful to do this in another way.

In this section, we write $(M,g)$ instead of $(M_0,g_0)$ and $x$ instead of $x'$, and thus $\dim(M)=n-1$. We set $m=n-1$. Let $(M,g) \subset \subset (\widehat{M},g)$ where $(M,g)$ and $(\widehat{M},g)$ are simple. A compact manifold $(M,g)$ with boundary is simple if the boundary is strictly convex, and the exponential map at any point is a diffeomorphism onto $M$.

Let $z \in M^{int}$ and fix $\xi \in S_z^* M$. Denote by $\gamma_\xi(t)$ the geodesic starting at $z$ in codirection $\xi$. We wish to construct a quasimode $v_s$, $s=\tau+i\lambda$, concentrating near $\gamma_\xi$ of the form 
$$
v_s = e^{is\psi} a
$$
satisfying $\norm{(-\Delta-s^2)v_s}_{L^2(M)} = O(\tau^{-\infty})$ and $\norm{v_s}_{L^2(M)} = O(1)$. Here the notation that a quantity is $O(\tau^{-\infty})$ means that the quantity is $O(\tau^{-N})$ for each $N$ large enough with implied constants depending on $N$.

To do this, let $p(z,\xi) = \gamma_\xi(-\hat{\tau}(z,-\xi)) \in \partial \widehat{M}$ be the point where the geodesic $\gamma_\xi$ enters $\widehat{M}$ ($\hat{\tau}$ is the time when $\gamma_{-\xi}$ exists $(\widehat{M},g)$), and choose 
$$
\psi(x; z, \xi) = \text{dist}_{(\hat{M},g)}(x, p(z,\xi)).
$$
Then $\psi$ is smooth in $(x,z,\xi)$ when $x$ is near $M$ and $(z,\xi) \in S^* M$, since $p(z,\xi)\in \partial \widehat{M}$ stays away from $x$ (see e.g.~\cite{M01X}), and $\psi$ satisfies the eikonal equation $\abs{\nabla \psi} = 1$. We compute 
\begin{align*}
e^{-is\psi}(-\Delta-s^2) v_s &= s^2(\abs{\nabla \psi}^2-1)a -is (2\langle d\psi, d\,\cdot\, \rangle + \Delta \psi)a - \Delta a \\
 &= -2is La - \Delta a
\end{align*}
where $2L = 2\langle d\psi, d\,\cdot\, \rangle + \Delta \psi$. We formally write 
$$
a = \sum_{j=0}^\infty s^{-j}a_{-j}
$$
and require that 
\begin{align*}
L a_0 &= 0, \\
L a_{-1} &= \frac{i}{2} \Delta a_0, \\
L a_{-N} &= \frac{i}{2} \Delta a_{-(N-1)} \\
 &\vdots
\end{align*}
After solving for $a_{-j}$ recursively, the sum can be made to converge using Borel summation as in  Section~\ref{sec_quasimodes}.


To solve the transport equations near $M$, take polar normal coordinates $(r,\theta)$ centered at the point $p(z,\xi)$ where the geodesic $\gamma_\xi$ enters $\widehat{M}$. Then $\psi = r$, $L=\p_r+\p_r(\log\abs{g}^{1/4})$, and we have 
$$
Lu = f \Longleftrightarrow \partial_r (\abs{g}^{1/4} u) = \abs{g}^{1/4} f
$$
where $\abs{g} = \det(g(r,\theta))$. Using this, we first choose a solution $a_0(r,\theta) = \abs{g}^{-1/4} \chi(\theta)$ for some $\chi \in C^{\infty}(S^{m-1})$, where $\chi$ is supported near $\dot{\gamma_\xi}(t)$ at the time $t=-\hat{\tau}(z,-\xi)$. Since
$$
r = \psi(x;z,\xi), \quad \theta = \frac{1}{\psi(x;z,\xi)} \exp_{p(z,\xi)}^{-1}(x),
$$
we see that $a_0 = a_0(x;z,\xi)$ depends smoothly on $(x,z,\xi)$ for $x$ near $M$. Integrating in $r$, we successively obtain functions $a_{-1},a_{-2}, \ldots$ that are independent of $\tau$, smooth near $M$, depend smoothly on $(x,z,\xi)$, vanish when $\theta \notin \supp(\chi)$, and satisfy the required equations. Using Borel summation as in Section~\ref{sec_quasimodes} we may find $a_s=a_s(x;z,\xi)$, $C^\infty$ in its arguments, so that
$$
a_s\sim\sum_{j=0}^\infty s^{-j}a_{-j}
$$
and that $v_s=e^{is\psi}a_s$ satisfies
$$
(-\Delta-s^2)v_s = O_{L^2(M)}(\tau^{-\infty}).
$$

We have completed the construction of a quasimode concentrating near a geodesic. Now consider two geodesics: let $z \in M^{int}$ be a interior point on the (transversal) manifold, let $\xi_1, \xi_2 \in S_z^* M$ with $\xi_1 \neq \xi_2$, and let $v_{s_1}$ and $w_{s_2}$ be quasimodes as above concentrating near $\gamma_{\xi_1}$ and $\gamma_{\xi_2}$ and having the forms 
$$
v_{s_1} \sim e^{is_1 \psi_1}(a_0 + s_1^{-1} a_{-1}\cdots ), \qquad w_{s_2} \sim e^{i s_2 \psi_2}(b_0 + s_2^{-1}b_{-1}+\cdots)
$$
where $\psi_j=\psi(\, \cdot \, ;z,\xi_j)$ and $s_j = \tau + i \lambda_j$, $j=1,2$. We use two functions $\chi_1, \chi_2 \in C^{\infty}(S^{m-1})$ for the two quasimodes.

Now, since $(M,g)$ is simple, the two geodesics intersect only at $z$. Choosing $\chi_j$ with small enough support, we may arrange so that $v_{s_1} w_{s_2}$ is supported in any given neighborhood of $z$. Writing $y$ for normal coordinates in $(M,g)$ centered at $z$, the integral of interest reduces to
\begin{equation}\label{error_of_apprx}
\int_M \hat{f}(\lambda,\cdot) v_{s_1} w_{s_2} \,dV_g = \int_{\mR^{m}} \hat{f}(\lambda,\cdot)e^{i\tau \tilde{\psi}} \tilde{a}  \,dy = O(\tau^{-\infty})
\end{equation}
where 
$$
\tilde{\psi}(y;z,\xi_1,\xi_2) = \psi(y;z,\xi_1)+\psi(y;z,\xi_2)
$$
and
$$
\lambda=\lambda_1-\lambda_2
$$
and 
$$
\tilde{a}(y;z,\xi_1,\xi_2,\tau,\lambda_1,\lambda_2) = e^{-\lambda_1 \psi_1 - \lambda_2 \psi_2}a_{s_1}b_{s_2}
\abs{g(y)}^{1/2}.
$$
The function $\hat{f}(\lambda,\cdot)$ is independent of $z,\xi_1,\xi_2$. That the integral~\eqref{error_of_apprx} is indeed $O(\tau^{-\infty})$ uses the Carleman estimates in~\cite{DKSjU09} and the fact that $v_{s_j}$, $j=1,2$, are eigenfunctions up to an error of $O_{L^2(M)}(\tau^{-\infty})$. 


Let now $(z_0,\xi_0) \in S^* M$ be the point and direction of interest. We wish to consider two fixed covectors $(z_0,\zeta_1), (z_0,\zeta_2) \in S_{z_0}^* M$, with $\zeta_1 + \zeta_2$ pointing in the direction of $\xi_0$, which generate the two geodesics that we will use in the construction above. Note that if $\zeta_1 + \zeta_2 = t\xi_0$ for some $t > 0$, and if also $\zeta_1 \neq \xi_0$ then necessarily 
\[
\zeta_2 = 2 \langle \zeta_1, \xi_0 \rangle \xi_0 - \zeta_1, \qquad 0 < \langle \zeta_1, \xi_0 \rangle < 1.
\]
Conversely, if we fix any $(z_0,\zeta_1)$ with $0 < \langle \zeta_1, \xi_0 \rangle < 1$, then $\zeta_2$ defined above will satisfy $\zeta_1 + \zeta_2 = t\xi_0$ where $t = 2 \langle \zeta_1, \xi_0 \rangle > 0$.

Assume now that we have fixed the covector $(z_0,\xi_0)$ of interest and a covector $(z_0, \zeta_1)$ such that $0 < \langle \zeta_1, \xi_0 \rangle < 1$ (thus $\zeta_1 + \zeta_2 = t_0 \xi_0$ where $0 < t_0 < 2$), so that the initial two geodesics are generated by $(z_0, \zeta_1)$ and $(z_0, \zeta_2)$. For the characterization of the wave front set, we wish to make the two geodesics depend smoothly on $(z,\xi)$ near $(z_0,\xi_0)$. One way to do this is as follows. 
The proof of the lemma is included in the Appendix.

\begin{lemma} \label{lemma_crossing_geodesics_parametrization}
Let $(z_0,\xi_0) \in S_{z}^* M$ be the point and direction of interest, and let $\zeta_1, \zeta_2 \in S_{z_0}^* M$ satisfy $\zeta_1 + \zeta_2= t_0 \xi_0$ with $0 < t_0 < 2$. Then, there exists a neighborhood $U_S$ of $(z_0,\xi_0)$ in $S^*M$ and a smooth mapping $I:U_S\to S^*M\times S^*M$, with
$$
I(\xi)=(\omega_1(\xi), \omega_2(\xi))
$$ 
so that 
\begin{gather*}
\omega_1(\xi_0) = \zeta_1, \qquad \omega_2(\xi_0) = \zeta_2, \\
\omega_1(\xi) + \omega_2(\xi) = t_0 \xi.
\end{gather*}
\end{lemma}


Let now $(z_0,\xi_0)\in S^*M$ and let $\omega_j:U_S\to S^*M$, $j=1,2$, be the parametrization given by the lemma above. Let $(z,\xi) \in T^* M$ be a non-zero covector. We set 
$$
\hat{\xi}=\frac{\xi}{\abs{\xi}}.
$$
Assume that $(z,\hat{\xi})$ is sufficiently close to $(z_0,\xi_0)$, so that $(z,\hat{\xi})\in U_S \subset S^*M$.
Thus, for $(z,\xi)\in T^*M$ with $(z,\hat{\xi})$ near $(z_0,\xi_0)$ we can define 
$$
\psi(y ; z,\xi) =\frac{\abs{\xi}}{t_0} \tilde{\psi}(y;z,\omega_1(\hat{\xi}), \omega_2(\hat{\xi}))
$$
and 
$$
a(y ; z,\xi) = \tilde{a}(y;z,\omega_1(\hat{\xi}), \omega_2(\hat{\xi}),\abs{\xi}/t_0,\lambda_1,\lambda_2) \chi(z,\xi)
$$
where $\chi$ is a $C^{\infty}$ cutoff with $\chi = 1$ when $z$ is close to $z_0$, $\hat{\xi}$ is close to $\xi_0$ and $\abs{\xi} \geq \tau_0$, and $\chi = 0$ otherwise. The norm of $\xi$ plays now the role of $\tau$, and by writing $\tau = \abs{\xi}/t_0$, we know that  
$$
\int_{\mR^m} e^{i\psi(y;z,\xi)} a(y;z,\xi) \hat{f}(\lambda,y) \,dy = O(\abs{\xi}^{-\infty})
$$
for all $(z,\hat{\xi})$ near $(z_0,\xi_0)$ (notice that the integral is really over a small neighborhood of $z_0$).

Integrating over $\xi$ gives that 
\[
\int_{\mR^m} \int_{\mR^m} e^{i\psi(y;z,\xi)} a(y;z,\xi) \hat{f}(\lambda,y) \, dy \,d\xi \in C^\infty
\]
as a function of the $z$ variable since repeated differentiation in $z$ of the integrand  does not alter its decay properties in $|\xi|$ and hence its integrability.
It remains to check that the operator
\[
    A_{\psi}f =\int_{\mR^m} \int_{\mR^m} e^{i\psi(y;z,\xi)} a(y;z,\xi)  f(y) \,dy \,d\xi 
\]
is a pseudodifferential operator with a principal symbol which does not vanish at $(z_0,\xi_0)$.
To do so we will use the equivalence of phase functions, see for instance Theorem 3.2.1 in \cite{sogg}.
The main point is that $\psi$ is homogeneous of degree $1$ in $\xi$, and that $\psi$ is of the form 
\[
\psi(y;z,\xi) = (y-z) \cdot \xi + \mathcal{O}(\abs{z-y}^2\abs{\xi})
\] 
To verify the later, use Taylor's formula by noting that 
\begin{align*}
\nabla_y \psi(y;z,\xi)|_{y=z} &= \frac{\abs{\xi}}{t_0} \nabla_y \tilde{\psi}(y;z,\omega_1(\hat{\xi}),\omega_2(\hat{\xi}))|_{y=z} \\
 &= \frac{\abs{\xi}}{t_0}(\nabla_y \psi(y;z,\omega_1(\hat{\xi})) + \nabla_y \psi(y;z,\omega_2(\hat{\xi})))|_{y=z} \\
 &= \frac{\abs{\xi}}{t_0} (\omega_1(\hat{\xi}) + \omega_2(\hat{\xi}))=\abs{\xi}\hat{\xi} \\
 &= \xi.
\end{align*}
From the form of $\psi$, we also get
\begin{align*}
     |\nabla_{\xi} \psi(y;z,\xi)| \geq |z-y|- \mathcal{O}(\abs{z-y}^2) \geq c|z-y|
\end{align*}
if we shrink as we may the support of the amplitude $a$ in $y$ and $z$ in a small support of $z_0$. 
Moreover, by construction $\tilde{a}$ is a polyhomogeneous symbol in $(z,\xi)$ hence $a \in S^0$ because $\chi$ cuts off frequencies lower than $\tau_0$.
Theorem 3.2.1 in \cite{sogg} implies that $A_{\psi}$ is a pseudodifferential operator of order $0$ and 
\[
     A_{\psi} - a(z,z,D) \in Op(S^{-1})
\]
hence the principal symbol of $A_{\psi}$ is 
\[
   e^{-\lambda_1\psi_1-\lambda_2\psi_2}a_0 b_0 |g(z)|^{1/2} = e^{-\lambda_1\psi_1-\lambda_2\psi_2}\chi_1 \chi_2
\]
which doesn't vanish at $(z_0,\xi_0)$ if $\chi_1$ and $\chi_2$ are chosen to equal $1$ near $\dot{\gamma}_1(t)$ and $\dot{\gamma}_2(t)$ respectively. 
From the fact that $A_{\psi}$ is a pseudodifferential operator of order $0$ with principal symbol non vanishing at $(z_0,\xi_0)$ and from
     $$ A_{\psi}\hat{f}(\lambda,\cdot) \in C^{\infty} $$
we deduce that $(z_0,\xi_0)$ does not belong to the wave front set of  $\hat{f}(\lambda,\cdot)$.
\begin{remark}
      An alternate argument is to observe that multiplying 
$$
\int_{\mR^m} e^{i\psi(y;z_0,\xi)} a(y;z_0,\xi) \hat{f}(\lambda,y) \,dy = O(\abs{\xi}^{-\infty})
$$
by $\kappa(x) e^{-i\psi(x;z_0,\xi)}$, where $\kappa$ is a cutoff to a small neighborhood of $z_0$, and integrating over $\xi$ gives that 
\[
\int_{\mR^m} \int_{\mR^m} e^{i\varphi(x,y;\xi)} P(x,y;\xi) \hat{f}(\lambda,y) \,dy \,d\xi \in C^\infty.
\]
as a function of the $x$ variable with
\begin{align*}
\varphi(x,y,\xi) &:= -\psi(x;z_0,\xi) + \psi(y;z_0,\xi), \\
P(x,y,\xi) &:= \kappa(x) a(y;z_0,\xi).
\end{align*}
By the previous computations on the phase $\psi$ we have the non-degeneracy of the mixed hessian $\psi''_{y \xi}(z;z,\xi) =  \mathrm{Id}$, hence 
     $$ \det \psi''_{y \xi}(z;z_0,\xi) \neq 0 $$
for $z$ close to $z_0$ and we can use Kuranishi's trick. From the equality
\begin{align*}
     -\varphi(x,y,\xi) = (x-y) \cdot  \bigg(\int_0^1 \nabla_y \psi(tx+(1-t)y;z_0,\xi) \, dt\bigg) = (x-y) \cdot \eta
\end{align*}
the change of variables 
\[
     \eta = \int_0^1 \nabla_y \psi(tx+(1-t)y;z_0,\xi) \, dt
\]
leads to
\[
\int_{\mR^m} \int_{\mR^m} e^{-i(x-y) \cdot \eta} \tilde{P}(x,y;\eta) \hat{f}(\lambda,y) \,dy \,d\eta \in C^\infty.
\]
And one can check just as in the previous alternate that the principal symbol of this pseudodifferential operator (which differs from $\tilde{P}(x,x,D)$ by an operator of order $-1$) does not vanish at $(z_0,\xi_0)$.
\end{remark}

\section{The strict Stefanov-Uhlmann regularity condition}\label{sec_Stefanov_Uhlmann}
We now shift our attention from simple (transversal) manifolds to a more general class of manifolds satisfying the strict Stefanov-Uhlmann regularity condition of Definition~\ref{SU_compl}. As in the case of simple manifolds, that we studied in the previous section, we wish to probe the singularities of $\hat{f}(\lambda,\cdot)$ at any given point $z_0\in M$ on the transversal manifold $M$ to any given direction $\xi_0\in S^*_{z_0}M$. To simplify the notation, we denote the direction of interest $(z_0,\xi_0)$ by $(x,\eta)\in S^*_xM$ in this section. 

To probe the singularities of $\hat{f}(\lambda,\cdot)$, our technique in the following sections requires that the direction of interest $\eta\in S^*_xM$ is generated by a pair $(\gamma_1,\gamma_2)$ of nontangential geodesics with
$$
\dot{\gamma}_1(0)+\dot{\gamma}_2(0)=t_0\eta
$$
that do not intersect each other outside $x$. We chose to call such pairs admissible geodesic pairs in  Definition~\ref{transveral_geos}. 
Additionally, our technique will require that not just $\eta$ is generated by an admissible geodesic pair, but that $\eta$ has a neighborhood in $S^*M$ generated by such a pairs. 
In this section we show that the strict Stefanov-Uhlmann condition at $\eta\in S^*M$ is sufficient for these requirements to hold true. 

Let $\eta\in S_x^*M$ be a given direction, and let $H=\{\eta\}^\bot\subset T_x^*M$ be the orthogonal complement to $\eta$. Let $\xi^\bot\in S_x^*M\cap H$ be a unit vector in the orthogonal complement. Let $\eps>0$ and define
$$
\xi_1=\xi^\bot+\eps\eta \mbox{ and } \xi_2=-\xi^\bot+\eps\eta.
$$
Then, we have that
\begin{equation}\label{generates_eta}
\xi_1+\xi_2=2\eps\eta \mbox{ and }
\hat{\xi}_1+\hat{\xi}_2=t_0\eta,
\end{equation}
where $\hat{\xi}_i=\xi_i/\abs{\xi_i}\in S^*M$ and $t_0=t_0(\eps)$. We also have
\begin{equation}\label{perturbation_close}
\|\hat{\xi}_1-\xi^\bot\|\sim \eps \mbox{ and } \|-\hat{\xi}_2-\xi^\bot\|\sim \eps.
\end{equation}
Here we can take $\|\cdot\|$ to be for example the Sasaki metric on $S^*M$.
The above means that $\eta$ is generated by the unit covectors $\hat{\xi}_i$, $i=1,2$, with the property that $\xi_1,-\xi_2$ are close to $\xi^\bot$. 

Now, referring to the parametrization $I$ of Lemma~\ref{lemma_crossing_geodesics_parametrization}, we have that $\eta$ has a neighborhood generated by pairs $I(\xi)=(\omega_1(\xi),\omega_2(\xi))$ of unit covectors, with $\xi$ close to $\eta$. The parametrization of Lemma~\ref{lemma_crossing_geodesics_parametrization} is continuous and $\omega_i(\eta)=\xi_i$. By the above we have that $\xi_1$ and $\xi_2$ can be chosen to be arbitrarily close to $\xi^\bot$ and $-\xi^\bot$ (by making $\eps$ smaller). Thus for any small neighborhood $U$ of $\xi^\bot$ in $S^*M$ there is a neighborhood $U_S$ of $\eta$ such that
\begin{equation}\label{nhood_of_eta_to_nhood_of_bot}
I(U_S)\subset U\times -U.
\end{equation}
We also have that $\omega_1(\xi)\neq \omega_2(\xi)$ for $\xi \in U_S$. We use these observations in the lemma below.

We recall the strict Stefanov-Uhlmann regularity condition (Definition~\ref{SU_compl}) in the setting of this section:
\begin{enumerate}
 \item[(a)] The geodesic $\gamma_{\xi^\bot}$ is nontangential and defined on the interval $[-t_{\mathrm{in}},t_{\mathrm{out}}
 ]$.
 \item[(b)] The graph $\gamma_{\xi^\bot}([-t_{\mathrm{in}},t_{\mathrm{out}}])$ of $\gamma_{\xi^\bot}$ contains no points conjugate to $x=\pi(\xi^\bot)$.
 \item[(c)] The geodesic $\gamma_{\xi^\bot}$ does not self-intersect for any time $t\in [-t_{\mathrm{in}},t_{\mathrm{out}}]$.
\end{enumerate}
We show next that this condition is enough for $\eta$ to have a neighborhood generated by admissible geodesic pairs. 

\begin{lemma}
 Let $(M,g)$ be a compact manifold with boundary, let $x\in \text{Int}(M)$ and $\eta\in S_x^*M$, and assume that $M$ satisfies the strict Stefanov-Uhlmann regularity condition at $\eta$. Then there exists a neighborhood of $\eta$ in $S^*M$ which is generated by admissible geodesic pairs.
\end{lemma}
\begin{proof}
 Let $\eta,\xi^\bot\in S_x^*M$ be as in the definition of the strict Stefanov-Uhlmann regularity condition. It is sufficient to show that if $(\xi_1,\xi_2)\in S^*M\times S^*M$ is any pair of not equal unit covectors with $\pi(\xi_1)=\pi(\xi_2)$, with $\xi_i$ both sufficiently close to $\xi^\bot$ in $S^*M$, then the geodesic pair $(\gamma_{\xi_1}, \gamma_{\xi_2})$ is an admissible geodesic pair. Indeed, if this is the case, then the covector pairs $I(\xi)=(\omega_1(\xi),\omega_2(\xi))$, with $\xi$ close to $\eta$, that generate a neighborhood of $\eta$, are such that the corresponding geodesic pairs $(\gamma_{\omega_1(\xi)},\gamma_{\omega_2(\xi)})$ are admissible. (Here we used the simple remark that if a geodesic pair $(\gamma_{\omega_1(\xi)},\gamma_{-\omega_2(\xi)})$ is admissible, so is $(\gamma_{\omega_1(\xi)},\gamma_{\omega_2(\xi)})$. We also used the fact that $\omega_1(\xi),-\omega_2(\xi)$ can be chosen to belong to an arbitrary small neighborhood of $\xi^\bot$ by shrinking $U_S$ in~\eqref{nhood_of_eta_to_nhood_of_bot}.)
 
 Let us first reduce all considerations to a sufficiently small neighborhood of $\xi^\bot$ in $S^*M$ such that all the geodesics corresponding to its covectors hit the boundary transversally, and that this happens after a time at most $T<\infty$ (forward and backward in time). This is possible since geodesics depend smoothly on their initial values and $\gamma_{\xi^\bot}$ itself is nontangential. Let us denote $\gamma=\gamma_{\xi^\bot}$.
 
 To show that there is a neighborhood $U$ of $\xi^\bot$ in $S^*M$ such that for all $\xi_1,\xi_2\in U$ with $\xi_1\neq \xi_2$ and $\pi(\xi_1)=\pi(\xi_2)$, the geodesics do not intersect outside $x$, we assume the opposite: Let $(\gamma_{\xi_1}^k,\gamma_{\xi_2}^k)$ be a sequence of geodesic pairs that intersect at times $t_i^k$, 
 \begin{equation}\label{intersect_eq}
 \gamma_{\xi_1^k}(t_1^k)=\gamma_{\xi_2^k}(t_2^k),
 \end{equation}
 with at least one of $t_1^k$ and $t_2^k$ nonzero, $\xi_1^k\neq \xi_2^k$,
 $$
 x_k:=\pi(\xi_1^k)=\pi(\xi_2^k) 
 $$
 and
 $$
 \|\xi_i^k-\xi^\bot\|_{S^*M}<\frac 1k, \ i=1,2.
 $$
 The times $-t_{\mathrm{in}}(\xi_i^k)$ and $t_{\mathrm{out}}(\xi_i^k)$ when the geodesic $\gamma_{\xi_i^k}$ enters and exits $M$ converge to the entrance and exit times $-t_{\mathrm{in}}$ and $t_{\mathrm{out}}$ of $\gamma$ as $k\to \infty$.
 Since $M$ is compact it has positive injectivity radius $\text{Inj}(M)>0$. (Here we have extended $M$ to a closed manifold to speak about $\text{Inj}(M)$. Note that $x_k\in \text{Int}(M)$ for $k$ large so that the boundary will not cause any complications here.) Thus we have
 $$
 \abs{t_1^k}\geq \text{Inj}(M) \mbox{ or } \abs{t_2^k}\geq \text{Inj}(M)
 $$
 for all $k$. Otherwise $\gamma_{\xi_1^k}$ and $\gamma_{\xi_2^k}$ would intersect at a geodesic ball centered at $x_k$. By passing to a subsequence, we may assume without loss of generality that
 $$
 \abs{t_1^k}\geq \text{Inj}(M).
 $$
 
 Since the intersection times $t_i^k$ belong to a compact interval $[-T,T]$, we may pass to another subsequence so that 
 $$
 t_i^k\to t_i, \mbox{ as } k\to\infty, \quad i=1,2.
 $$
 Since $\xi_i^k\to\xi^\bot$ in $S^*M$, we have by taking limit of~\eqref{intersect_eq}
 $$
 \gamma(t_1)=\gamma(t_2).
 $$
 (Recall that we denote $\gamma=\gamma_{\xi^\bot}$.) Since $\gamma$ by assumption has no self-intersections, we have
 \begin{equation}\label{same_time}
 s:=t_1=t_2\in [-t_{\mathrm{in}},t_{\mathrm{out}}].
 \end{equation}
 Since $\abs{t_1}\geq \text{Inj}(M)$, we have $s\neq 0$.

 Now, since by assumption $\gamma$ has no points conjugate to $x=\pi(\xi^\bot)$, we have that there is a neighborhood $U\subset TM$ of $s\xi^\bot$ such that the ``bundle version'' of the exponential map
  \begin{align*}
 E&:U\to M\times M, \\
 E(V)&=(\pi(V),\exp_{\pi(V)}(V)), \quad V\in U
 \end{align*}
 is a diffeomorphism (see e.g.~\cite[Lemma 5.12]{Lee97}). This follows by noting that the differential $DE$ of $E$ at $s\xi^\bot\in T^*M$ is of the form
 $$
 \left[
	\begin{array}{cccc}
	I_{n\times n} & 0 \\
	\# & D\exp_x \\
	\end{array}\right].
 $$
 Here $D\exp_x$ is the differential of the standard exponential map $\exp_x :T_xM\to M$, $x=\pi(s\xi^\bot)$, which is invertible near $s\xi^\bot \in T_xM$.
 We have that $t_i^k\xi_i^k\in U$ for $k$ large enough and for $i=1,2$, since  
 $$
 \xi_1^k,\xi_2^k\to \xi^\bot \mbox{ and } t_1^k,t_2^k\to s.
 $$
 But now we have that
 $$
 E(t_1^k\xi_1^k)=(x_k,\exp_{x_k}(t_1^k\xi_1^k))=(x_k,\exp_{x_k}(t_2^k\xi_2^k))=E(t_2^k\xi_2^k)
 $$
 for $k$ large enough with $t_1^k\xi_1^k\neq t_2^k\xi_2^k$. This is a contradiction to the exponential map $E$ being injective on $U$. 
 
 Thus there exists a neighborhood $U\times U$ of $(\xi^\bot,\xi^\bot)$ in $S^*M\times S^*M$ such that geodesics corresponding to pairs of its covectors intersect the boundary transversally. Moreover, if the covectors of a pair in $U\times U$ are not equal, the corresponding pair of geodesics do not intersect outside their starting point.
 
 We are left to show that we may shrink $U$, if necessary, so that the geodesics corresponding to its covectors do not self-intersect. We assume the opposite: there is a sequence $\xi^k\to \xi^\bot$ in $U\subset S^*M$ such that there are times $t_k<t_k'$ when the geodesic $\gamma_{\xi_k}$ intersects itself:
 $$
 \gamma_{\xi_k}(t_k)=\gamma_{\xi_k}(t_k').
 $$
 Since $\text{Inj}(M)>0$, we must have that $t_k'\geq t_k+ 2\,\text{Inj}(M)$. Otherwise the geodesic loop $\gamma_{\xi_k}:[t_k,t_k']\to M$ would belong to a geodesic ball of radius $\text{Inj}(M)$.
 
 Since $t_k,t_k'$ belongs to a compact time interval $[-T,T]$, we may pass to a subsequence so that
 $$
 t_{k}\to t \mbox{ and } t_{k}'\to t'.
 $$
 Since also $\xi_{k}\to \xi^\bot$, we have that 
 $$
 \gamma_{\xi^\bot}(t)=\gamma_{\xi^\bot}(t'),
 $$
 with $t'-t\geq 2\,\text{Inj}(M)>0$. Since $\gamma_{\xi^\bot}$ by assumption has no self-intersections this is a contradiction. Thus we may redefine $U$ so that it has all the required properties. This concludes the proof.
 \end{proof}
 
Combining our results we record the following. The proof is just the combination of the previous lemma and Lemma~\ref{lemma_crossing_geodesics_parametrization}.
\begin{prop}\label{parametrization_by_admissible}
 Let $(M,g)$ be a compact manifold with boundary, let $\eta\in S_x^*M$, and assume that $M$ satisfies the strict Stefanov-Uhlmann regularity condition at $\eta$.
 Then there is a neighborhood $U_\eta\subset S^*M$ of $\eta$ and a smooth mapping $I:U_\eta\to S^*U\times S^*U$ such that
 $$
 I(\xi)=(\omega_1(\xi), \omega_2(\xi)), \quad \omega_1(\xi) + \omega_2(\xi) = t_0 \xi 
$$
with
$$
(\gamma_{\omega_1(\xi)},\gamma_{\omega_1(\xi)})
$$
an admissible geodesic pair. Here $0<t_0<2$ is constant.
\end{prop}

This result will imply that by using Gaussian beams traveling along $\gamma_{\omega_1(\xi)}$ and $\gamma_{\omega_2(\xi)}$ we can detect singularities in the direction of $\eta$.

\section{Gaussian beam quasimodes} \label{sec_quasimodes}
In this section we discuss Gaussian beam quasimodes. Gaussian beam quasimodes are approximate eigenfunctions for the Laplace-Beltrami operator $\Delta_g$ on the transversal manifold $(M,g)$, $\dim(M)=m=n-1$. These were studied in~\cite{DKLS} and we recall some facts from there. 

We will first recall the Gaussian beam quasimode construction on any compact transversal manifold $(M,g)$ with boundary and for any nontangential geodesic $\gamma$. After this we will introduce smooth parameterizations for the quasimodes, and products of quasimodes, by cotangent vectors $\xi\in T^*M$.

\begin{prop} \label{prop_gaussianbeam_quasimode}
Let $\gamma: [0,L] \to M$ be a nontangential geodesic, and let $\lambda \in \mR$. There is a family of functions $(\hat{v}_s) \subset C^{\infty}(M)$, where $s = \tau+i\lambda$ and $\tau \geq 1$, such that 
\[
\norm{(-\Delta_{g} - s^2) \hat{v}_s}_{L^2(M)} = O(\tau^{-\infty}), \quad \norm{\hat{v}_s}_{L^2(M)} = O(1)
\]
as $\tau \to \infty$.
\end{prop}

This result was proved in \cite[Section 3]{DKLS} with $O(\tau^{-K})$ error estimates for large $K > 0$. We recall now the construction and also the standard extension to $O(\tau^{-\infty})$ error estimates using Borel summation.

To describe the quasimode $\hat{v}_s$, first embed $(M,g)$ in some closed manifold $(\widehat{M},g)$ and extend $\gamma$ as a unit speed geodesic in $\widehat{M}$. Since $\gamma$ is nontangential, there is an $\eps > 0$ such that $\gamma(t) \in \widehat{M} \setminus M$ for $t \in [-2\eps,0) \cup (L,L+2\eps)$. Assume for simplicity that $\gamma|_{[-2\eps,L+2\eps]}$ does not self-intersect. (For the general case see~\cite{DKLS}.)

Choose an orthonormal frame at $\gamma(0)$ with its first vector as $\dot{\gamma}(0)$. Then there exists a set of Fermi coordinates $(t,y)$ on a \emph{$\delta'$-neighborhood} 
\begin{equation}\label{delta_nhood}
U_{\delta'} = \{(t,y) \,;\, -2\eps < t < L+2\eps, \ \ \abs{y} < \delta' \}
\end{equation}
of the $\gamma$ curve for some $\delta' > 0$. In these coordinates the geodesic curve $\gamma(t)$ is mapped to $(t,0)$, and 
\[
g^{jk}|_{y=0} = \delta^{jk}, \ \ \partial_i g^{jk}|_{y=0} = 0.
\]
We write $x = (t,y)$ where $t = y^1$ and $y = (y^2, \ldots, y^{m})$.

The quasimode will have the form 
\[
\hat{v}_s = e^{is\widehat{\Theta}} \hat{a}_s.
\]
The phase function $\widehat{\Theta}$ was constructed in \cite{DKLS} and is given in Fermi coordinates  by the expression 
\[
\widehat{\Theta}(t,y) = t + \frac{1}{2} H(t)y \cdot y + \widetilde{\Theta}(t,y),
\]
where $(t,y)\in U_{\delta'}$.
The smooth $(m-1)\times (m-1)$ matrix function $H(t)$ solves the Riccati equation 
\begin{equation}\label{Riccati}
\dot{H}(t) + H(t)^2 = F(t), \qquad H(0) = H_0
\end{equation}
on the transversal manifold $M$, $\dim(M)=m=n-1$. Here $F(t)$ is a smooth matrix function involving second derivatives of the metric $g$, and $H_0$ is some complex symmetric matrix with $\mathrm{Im}(H_0)$ positive definite. Then $\mathrm{Im}(H(t))$ stays positive definite for all $t$~\cite[Lemma 2.56]{KKL}. Also, $\widetilde{\Theta}(t,y)= O(\abs{y}^3)$.

Fix a function $\eta \in C^{\infty}_c(\mR)$ with $0\leq \eta\leq 1$, $\eta(t) = 1$ for $\abs{t} \leq 1$ and $\eta(t) = 0$ for $\abs{t} \geq 2$, and let $(\lambda_j)_{j=0}^{\infty}$ be a sequence with $0 < \lambda_0 < \lambda_1 < \ldots \to \infty$. As the amplitude $\hat{a}_s$, we choose the function 
\[
\hat{a}_s(t,y) = \tau^{\frac{m-1}{4}} \sum_{j=0}^{\infty} (1-\eta(\tau/\lambda_j)) s^{-j} \hat{a}_{-j}(t,y) \chi(y/\delta').
\]
The smooth functions $\hat{a}_{-j}(t,y)$, $j \geq 0$, were constructed in \cite{DKLS} as solutions of transport equations, and they are independent of $s$. For any fixed $s$ the above sum is finite, and thus $\hat{a}_s \in C^{\infty}(M)$. We choose the numbers $\lambda_j$ to be so large so that for each $j \geq 0$ and for $\tau \geq 2$, 
\begin{equation} \label{lambdaj_choice}
\max_{0 \leq l \leq j} \ \norm{(1-\eta(\tau/\lambda_j)) \nabla^l [\hat{a}_{-j}(t,y) \chi(y/\delta')]}_{L^{\infty}(M)} \leq \tau.
\end{equation}
Here we used that $\abs{1-\eta(t)}\leq \abs{t}$ for all $t\in \mathbb{R}$.
The estimate \eqref{lambdaj_choice} implies that, for $l \geq 0$, 
\[
\tau^{-\frac{m-1}{4}} \abs{\nabla^l \hat{a}_s} = \left\lvert \left[ \sum_{j=0}^l + \sum_{j=l+1}^{\infty} \right] (1-\eta(\tau/\lambda_j)) s^{-j} \nabla^l [ \hat{a}_{-j}(t,y) \chi(y/\delta') ] \right\rvert \leq C_l
\]
with $C_l$ independent of $s$. Here we used direct $L^\infty$ bound for the first sum. For the second sum we used~\eqref{lambdaj_choice} and summation of geometric series. Similarly, if $N \geq 0$ and $0 \leq l \leq N$ we have 
\begin{multline*}
\abs{\nabla^l (\tau^{-\frac{m-1}{4}} \hat{a}_s - \sum_{j=0}^N s^{-j} \hat{a}_{-j} \chi(y/\delta'))} \leq \Big\lvert \sum_{j=0}^N \eta(\tau/\lambda_j) s^{-j} \nabla^l [ \hat{a}_{-j}(t,y) \chi(y/\delta') ] \Big\rvert  \\
 + \Big\lvert \sum_{j=N+1}^{\infty} (1-\eta(\tau/\lambda_j)) s^{-j} \nabla^l [ \hat{a}_{-j}(t,y) \chi(y/\delta') ] \Big\rvert
\end{multline*}
which is seen to be $\leq C_N \tau^{-N}$ by using the fact that $\eta(t) \leq C_N t^{-N}$ for all $t\in \R$ in the first sum and the estimate \eqref{lambdaj_choice} in the second sum. This shows that 
\[
\tau^{-\frac{m-1}{4}} \hat{a}_s(t,y) \sim \sum_{j=0}^{\infty} s^{-j} \hat{a}_{-j}(t,y) \chi(y/\delta')
\]
in the sense of an asymptotic sum of semiclassical symbols.

From the amplitude function $\hat{a}_s$, we see that $\hat{v}_s$ is supported on $U_\delta'$. For future reference, we remark that choosing $\delta'$ smaller does not change the construction. The same functions $\widehat{\Theta}$ and $\hat{a}_{-j}$ will do in that case.

The above construction and \cite{DKLS} (see the proof of Proposition 3.1) imply that $\norm{\hat{v}_s}_{L^2(M)} \leq C$ with $C$ independent of $s$. The error term $f = (-\Delta-s^2)\hat{v}_s$ describing how $\hat{v}_s$ departs from a true eigenfunction is of the form $f = f_1 + f_2$, where 
\begin{align*}
f_1 &= (-\Delta-s^2) (e^{is\widehat{\Theta}} \tau^{\frac{m-1}{4}} (\hat{a}_0 + s^{-1} \hat{a}^{-1} + \ldots + s^{-N} a_{-N}) \chi(y/\delta')), \\
f_2 &= (-\Delta-s^2) (e^{is\widehat{\Theta}} [\hat{a}_s - \tau^{\frac{m-1}{4}} (\hat{a}_0 + s^{-1} \hat{a}^{-1} + \ldots + s^{-N} a_{-N}) \chi(y/\delta')]).
\end{align*}
Here $\norm{f_1}_{L^2(M)} \leq C_N \tau^{\frac{3-N}{2}}$ as in \cite[Proof of Proposition 3.1]{DKLS}, and $\norm{f_2}_{L^2(M)} \leq C_N \tau^{2-N}$ by the above symbol estimates. After replacing $N$ by $2N+3$, we obtain that for any $N$ there is $C_N$ with 
\[
\norm{(-\Delta-s^2) \hat{v}_s}_{L^2(M)} \leq C_N \tau^{-N}
\]
as required.
\subsection{Parametrization of quasimodes}
Next we parametrize locally the quasimodes and products of pairs of quasimodes by cotangent vectors $(z,\xi)\in T^*M\setminus \{0\}$ corresponding to nontangential geodesics. This will be done in four steps. 
In steps (3) and (4) we assume $(z_0,\xi_0)$ is generated by a pair of nontangential geodesics.

\medskip

\noindent {\bf (1)} We define the function $\hat{v}_s$ by
  $$
  \hat{v}_s=\hat{v}_s(\xi,F,x)=e^{is\widehat{\Theta}(\xi,F,x)}\hat{a}_s(\xi,F,x), 
  $$
  with
  $$
  \xi\in  S_z^*M,\quad  F\in FO_\xi(M), \quad x\in M
  $$
  as the Gaussian beam quasimode with complex energy $s^2$ of Proposition~\ref{prop_gaussianbeam_quasimode} above with the following data:

The covector $\xi \in S_z^*M$ defines a non-tangential geodesic $\gamma_\xi$ with initial direction $\xi$. The orthonormal coframe 
$$
F\in FO_\xi(M):=\{F=(F_1,\ldots,F_n)\in FO_z^*(M): F_1=\xi \}
$$ is an orthonormal coframe for $T_z^*M$ with its first covector as $\xi$. Here $FO_z^*(M)$ denotes the set of coframes at $z$. The frame $F$ is used to define Fermi coordinates on an open $\delta'$-neighborhood of $\gamma_\xi$ as in~\eqref{delta_nhood}. We choose $H_0=i I_{\tiny{(n-2)\times (n-2)}}$ as the initial value for the Riccati equation~\eqref{Riccati} in the Fermi coordinates around $\gamma_\xi$. Here $i$ is the imaginary unit. 

The ``hat'' on top of the functions $v_s, \Theta$ and $a_s$ are here to differentiate them from the functions we introduce next.

\medskip

\noindent {\bf (2)} We set for $\xi\in T_z^*M\setminus \{0\}$
  $$
  \hat{\xi}=\frac{\xi}{\abs{\xi}_{g(z)}}
  $$
  and define the function $v_s$ by
  \begin{equation}\label{polyhomogenization_of_gb}
  v_s(\xi,F,x)=e^{i\tau\Theta(\xi,F,x)}e^{-\lambda\widehat{\Theta}(\hat{\xi},F,x)}a_s(\xi,F,x), 
  \end{equation}
  where $ F\in FO_{\hat{\xi}}(M)$, $x\in M$, and
  $$
  \Theta(\xi,F,x):=\abs{\xi}_g\widehat{\Theta}(\hat{\xi},F,x),
  $$
  and
    \begin{multline*}
  \quad  \tau^{-\frac{n-1}{4}}a_{s}(\xi,F,x):=\tau^{-\frac{n-1}{4}}\hat{a}_{s_{\abs{\xi}}}(\hat{\xi},F,x) \\ \sim \hat{a}_0(\hat{\xi},F,x)+s_{\abs{\xi}}^{-1}\hat{a}_{-1}(\hat{\xi},F,x)+ s_{\abs{\xi}}^{-2}\hat{a}_{-2}(\hat{\xi},F,x)+\cdots
  \end{multline*}
  with
  $$
  s_{\abs{\xi}}=\tau\abs{\xi}+i\lambda.
  $$
  
  We call the function $v_s$ the polyhomogenization of a Gaussian beam $\hat{v}_s$ by $\abs{\xi}_g$. This function is just the result of setting the complex frequency $s$ as the $x$-independent constant
  $
  \tau\abs{\xi}_{g(z)}+i\lambda
  $
  instead of $\tau+i\lambda$ in the construction of the Gaussian beam above. 
  
  The polyhomogenization is just for technical purposes since we need our presentation to be compatible with the work~\cite{WZ}, whose results we wish to use without modifications.

  \medskip

\noindent {\bf (3)} Let $(z_0,\xi_0)\in T^*M\setminus\{0\}$ be a point and direction of interest such that $(z_0,\xi_0)$ has a neighborhood generated by nontangential geodesics. Next we consider the products of Gaussian beams of complex frequency
  $$
  s_j=\tau|\xi|_g+i\lambda_j, \quad j=1,2.
  $$
  This results in a function of the form
  $$
  e^{i\tau\Phi(\xi,\overline{F},x)}A_{\tau,\lambda_1,\lambda_2}(\xi,\overline{F},x).
  $$
  with phase function $\Phi$ as
  $$
  \Phi(\xi,\overline{F},x)=\Theta(\abs{\xi}\omega_1(\hat{\xi}),F_1,x)+\Theta(\abs{\xi}\omega_2(\hat{\xi}),F_2,x)
  $$
  and amplitude function $A$ as
  \begin{multline}
  A_{\tau,\lambda_1,\lambda_2}(\xi,\overline{F},x)=e^{-\lambda_1\hat{\Theta}(\omega_1(\hat{\xi}),F_1,x)-\lambda_2\hat{\Theta}(\omega_2(\hat{\xi}),F_2,x)} \\ \times a_{s_1}(\abs{\xi}\omega_1(\hat{\xi}),F_1,x) a_{s_2}(\abs{\xi}\omega_2(\hat{\xi}),F_2,x).
  \end{multline}
  Here the variables are as
  $$
  \xi\in \mathcal{U}\subset T^*M, \quad \overline{F}=(F_1,F_2), F_i\in FO_{\omega_i(\hat{\xi})}(M), \quad  \mbox{and } x\in M,
  $$
  where $\omega_i(\hat{\xi})$, $i=1,2$, are the component functions of the parametrization introduced in Lemma~\ref{lemma_crossing_geodesics_parametrization} satisfying
  \begin{equation}\label{crossing_geo}
  \omega_1(\hat{\xi})+\omega_2(\hat{\xi})=t_0\hat{\xi}.
  \end{equation}
  Above $\mathcal{U}$ is the conic extension of the neighborhood $U_S$ of $\hat{\xi}_0$ introduced in Lemma~\ref{lemma_crossing_geodesics_parametrization}: 
  $$
  \mathcal{U}=\{t\xi\in T^*M: \xi\in U_S, t\in (0,\infty)\}.
  $$
  Thus $\mathcal{U}$ is a neighborhood of $(z_0,\xi_0)$ in $T^*M$.

  \medskip

\noindent {\bf (4)} Lastly we parametrize the orthonormal frames $F_i\in FO_{\omega_i(\hat{\xi})}(M)$, $i=1,2$, by unit cotangent vectors.
We define the function $u_\tau(\xi,x)$ that corresponds to products of pairs of Gaussian beams parametrized by $\xi$ as 
  \begin{equation}\label{def_of_u_tau}
  u_\tau(\xi,x)=e^{i\tau\Phi(\xi,F(\hat{\xi}),x)}A_{\tau,\lambda_1,\lambda_2}(\xi,F(\hat{\xi}),x),
  \end{equation}
  where
  $$
    \xi\in  \mathcal{U}\subset T^*M, \quad x\in M.
  $$
  Here 
  $$
  F(\hat{\xi})=\left(F(\omega_1(\hat{\xi})),F(\omega_2(\hat{\xi}))\right)
  $$ 
  is a parametrization of orthonormal frames given by Lemma~\ref{frame_parametrization} in the appendix. The parametrization is defined on an open neighborhood of the direction of interest $(z_0,\hat{\xi}_0)$ in $S^*M$. By shrinking $U_S$ of Lemma~\ref{lemma_crossing_geodesics_parametrization}, we can take $\mathcal{U}$ to be the conic extension of $U_S$ as above.

 We record the following facts.
 \begin{lemma}\label{basics_of_param_prod}
The polyhomogenization $v_s(\xi,F,x)$ of the Gaussian beam $\hat{v}_s(\xi,F,x)$ is an approximate eigenfunction of complex frequency
  $$
  s=\tau\abs{\xi}_g+i\lambda
  $$
  in the sense that it satisfies:
  $$
  \norm{(-\Delta-s^2)v_s}_{L^2(M)} = O((\tau\abs{\xi})^{-\infty}) \mbox{ and } \norm{v_s}_{L^2(M)} = O(1)
  $$
  for $\tau\abs{\xi}$ large.
  
  Let $(z_0,\xi_0)\in T^*_{z_0}M\setminus\{0\}$ be a point and direction of interest. Then there is a neighborhood $\mathcal{U}$ of $(z_0,\xi_0)$ in $T^*M$ such that $u_\tau(\xi,x)$ is defined on $\mathcal{U}\times M$ and depends smoothly on its variables $\xi\in \mathcal{U}$, $x\in M$.


\end{lemma}
\begin{proof}
The first claim is just the construction of Gaussian beam of Proposition~\ref{prop_gaussianbeam_quasimode} with $\tau\abs{\xi}_{g(z)}+i\lambda$ in place of $\tau+i\lambda$.

For the second one, we note that the parameterizations $\omega_i(\hat{\xi})$ and $F(\omega_i(\hat{\xi}))$ are smooth functions of $\xi\neq 0$. Likewise, the Fermi coordinates depend smoothly on this data. Gaussian beams in the corresponding Fermi coordinates are smooth in $x$, and thus smooth functions in general.

Moreover, the solution of the Riccati equation used in the construction of a Gaussian beam depends smoothly on its initial value $H_0$ and its metric dependent coefficients. This can be seen from~\cite[Lemma 2.56]{KKL} where the solution to the Riccati equation is given by a solution to a (linear) system of ordinary differential equations. Solutions to that type of equation depends smoothly on the coefficients of the equation and on the initial data (see e.g.~\cite[Section 1.6]{T_Book1}). The components $\hat{a}_{-j}$ of the amplitude functions used to define $u_\tau(\xi,x)$ are constructed by solving transport equations. These equations reduce to a system of ordinary differential equations for the ``$t$-dependent'' coefficients as can be seen from~\cite[Proposition 3.1]{DKLS}. Combining these facts proves the claimed smoothness in $\xi$ and $x$ variables.
\end{proof}

 \subsection{The phase and amplitude functions}
 We show next that the functions $\Phi(\xi,F(\xi),x)$ and $A_{\tau,\lambda_1,\lambda_2}(\xi,F(\xi),x)$ defined on $\mathcal{U}\times M$ are admissible phase and amplitude functions in the sense of~\cite{WZ} if $\mathcal{U}$ is generated by admissible geodesics. We begin with the phase function. 
 
 We make the following simplifications of notation:
 \begin{align*}
 \Phi(\xi,x):&=\Phi(\xi,F(\xi),x), \\
 A_{\tau}(\xi,x):&=A_{\tau,\lambda_1,\lambda_2}(\xi,F(\xi),x), \\
 \widehat{\Theta}_i(\xi,x):&=\widehat{\Theta}(\omega_i(\hat{\xi}),F(\omega_i(\hat{\xi})),x), \quad i=1,2.
 \end{align*}
 These notations are justified since the parameterizations of the orthonormal frames or the crossing geodesics (Lemma~\ref{lemma_crossing_geodesics_parametrization}) play no explicit role in what follows, and since $\lambda_i\in \R$, $i=1,2$, are fixed. Combining the first and the last, we write
 $$
 \Phi(\xi,x)=\abs{\xi}_{g(z)}(\widehat{\Theta}_1(\xi,x)+\widehat{\Theta}_2(\xi,x)), \quad \xi\in \mathcal{U}, \ z=\pi(\xi), \ x\in M.
 $$
 
 \begin{prop}\label{phase_conds}
 Let $(z_0,\xi_0)\in T^*M$ be a point and direction of interest that has a neighborhood generated by admissible geodesics. Then there is a neighborhood $\mathcal{U}\times U$ of $(\xi_0,z_0)$ such that the phase function $\Phi(\xi,x)$ of $u_\tau(\xi,x)$ is an admissible phase function in the sense of~\cite{WZ} on $\mathcal{U}\times U$:
\begin{enumerate}
 \item $\Phi$ is a polyhomogenous symbol of order one in $\xi$, 
 \smallskip
 \item $d_x\Phi|_\Delta = t_0\xi\cdot dx$, 
 \smallskip
 \item $\nabla_x^2\mbox{Im}\Phi|_\Delta \sim \langle \xi \rangle$,
 \smallskip
 \item $\Phi|_\Delta =0$
 \smallskip
 \item $\mbox{Im} \Phi \geq 0$
\end{enumerate}
Here $\Delta=\{(\xi,x)\in \mathcal{U}\times U; \pi(\xi)=x\}$ and $0<t_0<2$ is some constant and $\nabla_x^2$ is the Riemannian Hessian. 
\end{prop}
We remark that there is an additional assumption in~\cite[Definition 2.1]{WZ} that $\Phi$ is an elliptic symbol. This however can be omitted. We thank Jared Wunsch for clarifying this to us. The reason for the conditions above is that one can write the phase function in any local coordinates near the diagonal $\Delta$ as
\begin{equation}\label{local_xpr}
\Phi=t_0\xi\cdot (x-z)+\langle Q(x,\xi)(x-z),(x-z)\rangle,
\end{equation}
Here $Q$ is a symmetric matrix-valued symbol (depending on the used local coordinates) with $\text{Im}(Q)|_\Delta\sim \langle \xi \rangle$. When the conditions $(1)$-$(5)$ above hold, the formula~\eqref{local_xpr} follows by Taylor expanding in $x$ around $z=\pi(\xi)$. Compare with the part $(5)$ of the proof which follows.
\begin{proof}[Proof of Proposition~\ref{phase_conds}.]
Let $(z_0,\xi_0)\in T^*M$ and let $\Phi(\xi,x)$ be first defined on $\mathcal{U}\times U$ as explained in the beginning of this section. 
We will need to redefine both $\mathcal{U}$ and $U$ while we advance in the proof, and then restrict $\Phi(\xi,x)$ onto the redefined sets.

For $\xi\in \mathcal{U}$, the phase function reads
 $$
 \Phi(\xi,x)=\abs{\xi}_g\left(\hat{\Theta}(\omega_1(\hat{\xi}),F(\omega_1(\hat{\xi})),x)+\hat{\Theta}(\omega_2(\hat{\xi}),F(\omega_2(\hat{\xi})),x)\right).
 $$
 Recall that for a unit covector $\sigma\in S^*M$, the function 
 $$
 \hat{\Theta}(\sigma,F(\sigma),x)
 $$
 is the phase function of a Gaussian beam constructed for a nontangential geodesic $\gamma_\sigma$ in $(\sigma,F(\sigma))$-Fermi coordinates.
 
\medskip
\noindent {\bf (1)} The phase function $\Phi$ is $1$-(poly)homogeneous since 
$$
\xi\mapsto \hat{\Theta}(\omega_i(\hat{\xi}),F(\omega_i(\hat{\xi})),x), \quad i=1,2,
$$ 
is $0$-homogeneous by definition.

\medskip

\noindent {\bf (2)} Let us verify the three conditions of the behavior of $\Phi$ on the diagonal. Let $\hat{\xi}\in S_z^*U$ and $F(\hat{\xi})\in FO_{\hat{\xi}}(M)$. This data defines Fermi coordinates $(t,y)=(y^1,\ldots, y^{m})$ uniquely, where we denote $y^1=t$ and $y=(y^2,\ldots,y^{m})$. In these Fermi coordinates we have 
$$
\widehat{\Theta}(t,y)=t+\frac{1}{2}H(t)y\cdot y + \widetilde{\Theta}(t,y)
$$
for the phase function of the corresponding Gaussian beam, where 
$$
\widetilde{\Theta}(t,y)=O(\abs{y}^3).
$$
(See the beginning of this section.) Here we have omitted $\xi$-dependent quantities from the presentation to simplify the notation.

We first show that
\begin{equation}\label{osc_dir}
d_x\Theta(\xi,F,x)|_\Delta=\xi\cdot dx,
\end{equation}
for any $\xi\in T^*M\setminus\{0\}$, $F\in FO_{\hat{\xi}}(M)$. From this it follows that 
\begin{align*}
d_x\Phi(\xi,x)\big|_\Delta&=d_x\left(\Theta(\abs{\xi}\omega_1(\hat{\xi}),F(\omega_1(\hat{\xi})),x)+\Theta(\abs{\xi}\omega_2(\hat{\xi}),F(\omega_2(\hat{\xi})),x)\right)\big|_\Delta \\
&=(\abs{\xi}_{g}\omega_1(\hat{\xi})+\abs{\xi}_{g}\omega_2(\hat{\xi}))\cdot dx=t_0\xi\cdot dx
\end{align*}
since
$$
\omega_1(\hat{\xi})+\omega_2(\hat{\xi})=t_0\hat{\xi}
$$
by the definition of $\omega_i$, $i=1,2$, in Lemma~\ref{lemma_crossing_geodesics_parametrization}.

To show~\eqref{osc_dir}, we note that it is a pointwise equation in the $\xi$-variable. Thus we may fix $\xi\in T_z^*M\setminus\{0\}$ and calculate in $(\hat{\xi},F(\hat{\xi}))$-Fermi coordinates $y=(y^1,\ldots,y^{m})$, with $y^1=t$ and $y=(y^2,\ldots,y^{m})$. We have $y(z)=0$, and
\begin{align*}
d_y\Theta(t,y)|_\Delta&=\abs{\xi}_{g(0)}\frac{\p}{\p y^i}\Big(t+\frac{1}{2}H(t)y\cdot y + \widetilde{\Theta}(t,y)\Big)\Big|_{y=0}dy^i \\
&=\abs{\xi}_{g(0)}dt=\xi_idy^i.
\end{align*}
Here we have used the fact that in $(\hat{\xi},F(\hat{\xi}))$-Fermi coordinates $\hat{\xi}=dt$. The claim (3) follows.

\noindent {\bf (3)} We analyze the Hessian of $\Phi$ in a similar manner. Let $\xi\in T_z^*U\setminus\{0\}$. We show first that 
\begin{align}\label{posatort}
d_x^2\text{Im}(\Theta(\xi,F,x))|_\Delta=\abs{\xi}_g M^\bot(\hat{\xi}, x)\in T^2_0M,
\end{align}
where $M^\bot(\xi,\cdot)$ is a local $2$-tensor field on $M$, positive definite in the orthogonal complement 
$$
H_\xi=\{\omega\in T^*_zM: g(\omega,\xi)=0\}
$$ 
of $\xi$, and $M^\bot(\xi,\cdot)(\xi,\eta)=0$ for all $\eta\in T_z^*M$. From this it will then follow that
\begin{align*}
d_x^2\text{Im}\,\Phi(\xi,x)|_\Delta& =\abs{\xi}_{g} M^\bot(\omega_1(\hat{\xi}), x) +\abs{\xi}_{g}  M^\bot(\omega_2(\hat{\xi}), x) \\ 
&=\abs{\xi}_{g} M(\hat{\xi}, x),
\end{align*}
where $M(\xi,x)$ is a positive definite matrix field in the whole cotangent space $T^*_zU$ since $\omega_1(\hat{\xi})$ and $\omega_2(\hat{\xi})$ are not parallel.

Again, since the claim is pointwise in the $\xi$-variable, we calculate in $(\hat{\xi},F(\hat{\xi}))$-Fermi coordinates. We have
\begin{align*}
d_x^2\Theta(\xi,F(\hat{\xi}),x)\big|_\Delta
&=\abs{\xi}_{g(0)}\left(\frac{\p^2}{\p y^i\p y^j}\left(t+\frac{1}{2} H(t)y\cdot y + \widetilde{\Theta}(t,y)\right)\right)\Big|_{y=0} dy^i\otimes dy^j \\
&=\abs{\xi}_{g(0)}\sum_{i,j=2}^mH(t)_{ij}dy^i\otimes dy^j.
\end{align*}
Here we have used the fact that $d_x^2\text{Im}(\widetilde{\Theta}(t,y))=O(|y|)$, and that in Fermi coordinates the Christoffel symbols vanish, and metric is the identity matrix, on the corresponding geodesic. Thus, we have~\eqref{posatort}, and consequently $(4)$.

\noindent {\bf (4)} If $\xi\in T_z^*M\setminus\{0\}$ is fixed, then in the corresponding Fermi coordinates $y$ we have $y(z)=0$. The claim follows from the formula of the phase function of a Gaussian beam in Fermi coordinates.

\noindent {\bf (5)} This follows from $(2)$-$(4)$: Let $\xi\in \mathcal{U}$. We Taylor expand in local coordinates at $z=\pi(\xi)$ using $\Phi|_{x=z}=0$ and $\mbox{Im}(d_x\Phi|_{x=z})=0$. (Especially the later implies $(\nabla^2 \text{Im}\Phi)_{ij}=\p_i\p_j\text{Im}\Phi$ on $\{x=z\}$.) We have
$$
\mbox{Im} \Phi=\langle\nabla^2\mbox{Im} \Phi(x,\xi)|_{x=z} (x-z),(x-z)\rangle + O_\xi(|x-z|^3),
$$
where $\langle \xi\rangle/C_\xi\leq \nabla^2\mbox{Im}\Phi(x,\xi)|_{x=z} \leq \langle \xi\rangle C_\xi$. Since $\Phi$ is smooth in its variables we have that there is uniform $C>1$ such that $1/C \leq \nabla^2\mbox{Im}\Phi(x,\xi)|_{x=z}$ if we redefine $\mathcal{U}$ as $\{t\xi\in T^*M: \xi\in U_S, t\in (\frac 12 \abs{\xi_0},\infty)\}$. 

Thus we may write by using Taylor's theorem with a remainder as
$$
\mbox{Im} \Phi=\langle\nabla^2\mbox{Im} Q(x,\xi) (x-z),(x-z)\rangle,
$$
where $Q(x,\xi)$ is uniformly positive definite in its variables on a neighborhood of the diagonal $\Delta$. Shrinking $\mathcal{U}$ and $U$ further so that $\pi(\mathcal{U}) \times U$ belongs to this neighborhood, and contains $(z_0,z_0)$, gives $\mbox{Im} \Phi\geq 0$ on $\mathcal{U}\times U$.
\end{proof}

We continue with the amplitude function $A_\tau(\xi,x)$. We first show that this is a polyhomogenous symbol in $S_\phg^{\frac{m-1}{2},0}$ in the sense of Wunsch and Zworski~\cite{WZ}. Later on, we will multiply $A_\tau$ with suitable powers of $\tau$ and $\abs{\xi}$, so that after these multiplications, the result is in $S_\phg^{\frac{3m}{4},\frac m4}$ as required by~\cite{WZ}. (The $\tau$ dependence factor $\frac{m-1}{2}$ comes from multiplying the amplitudes of two Gaussian beams with powers of $\tau$ of $\frac{m-1}{4}$.)
\begin{prop}\label{amplitude_prop}
The amplitude function $A_\tau(\xi,x)$ on $\mathcal{U}\times U\times [\tau_0,\infty)$
\begin{align*}
A_\tau(\xi,x)&\sim e^{-\lambda_1\hat{\Theta}(\omega_1(\hat{\xi}),F(\omega_1(\hat{\xi})),x)-\lambda_2\hat{\Theta}(\omega_2(\hat{\xi}),F(\omega_2(\hat{\xi})),x)}\tau^{\frac{m-1}{2}} \\
&\times \sum_{j,l=0}^\infty s_{jl}\hat{a}_{-j}(\omega_1(\hat{\xi}),F(\omega_1(\hat{\xi})),x)\hat{a}_{-l}(\omega_2(\hat{\xi}),F(\omega_2(\hat{\xi})),x),
\end{align*}
where
$$
s_{jl}=\left(\frac{1}{\tau\abs{\xi}+i\lambda_1}\right)^j\left(\frac{1}{\tau\abs{\xi}+i\lambda_2}\right)^l
$$
is a polyhomogenous symbol in the class $S_\phg^{\frac{m-1}{2},0}$. Here $\tau_0$ is sufficiently large.
\end{prop}
\begin{proof}
We have the expansion
$$
s_{jl}=(\tau\abs{\xi})^{-(j+l)}-i\lambda_1(\tau\abs{\xi})^{-(j+l+1)}-i\lambda_2(\tau\abs{\xi})^{-(j+l+1)}+\cdots
$$
Since each $a_{-j}(\omega_i(\hat{\xi}),F(\omega_i(\hat{\xi})),x)$, $j=1,2,\ldots$, $i=1,2$, is continuous in its variables, and $0$-homogeneous in $\xi$, we have
\begin{align*}
s_{jl}&\hat{a}_{-j}(\omega_1(\hat{\xi}),F(\omega_1(\hat{\xi})),x)\hat{a}_{-l}(\omega_2(\hat{\xi}),F(\omega_2(\hat{\xi})),x) \\
&=(\tau\abs{\xi})^{-(j+l)}O_{j,l}(1)
+(\tau\abs{\xi})^{-(j+l+1)}O_{j,l}(1)+\cdots
\end{align*}
and we can write the asymptotic sum
$$
\sum_{j,l=0}^\infty s_{jl}\hat{a}_{-j}(\omega_1(\hat{\xi}),F(\omega_1(\hat{\xi})),x)\hat{a}_{-l}(\omega_2(\hat{\xi}),F(\omega_2(\hat{\xi})),x)
$$
by arranging the powers of $\tau\abs{\xi}$ as
$$
\sum_{j=0}^\infty\tau^{-j}\abs{\xi}^{-j}B_{-j},
$$
where each
$$
B_{-j}=B_{-j}(\xi,x)=O_{j}(1)
$$
in its both variables.

The factor
$$
e^{-\lambda_1\widehat{\Theta}(\omega_1(\hat{\xi}),F(\omega_1(\widehat{\xi})),x)-\lambda_2\hat{\Theta}(\omega_1(\hat{\xi}),F(\omega_1(\hat{\xi})),x)}
$$
is smooth in its variables and thus satisfies
 $$
 \abs{e^{-\lambda_1\hat{\Theta}(\omega_1(\hat{\xi}),F(\omega_1(\hat{\xi})),x)-\lambda_2\hat{\Theta}(\omega_1(\hat{\xi}),F(\omega_1(\hat{\xi})),x)}}\leq C
 $$
 on $\mathcal{U}\times U$.

Set $\tilde{n}=\frac{m-1}{2}$. It follows automatically that $A_\tau(\xi,x)$ satisfies
$$
\abs{A_\tau(\xi,x)-\tau^{\tilde{n}}(B_0+\cdots+\tau^{-j}B_{-j})}\leq C_j \tau^{-\tilde{n}-j-1}\abs{\xi}^{0-j-1}, \quad \text{for } \abs{\xi}>1.
$$
Thus 
$
A_\tau(\xi,x)
$
satisfies the growth condition of~\cite[Def. 2.3.]{WZ}. 
\end{proof}

\section{Recovery of singularities}\label{sec_fbi}
We apply Theorem 4.8 of~\cite{WZ} to prove our main theorem, Theorem~\ref{thm_main1}. That is, we show  that if $\xi_0\in T^*_{z_0}M_0$ has a neighborhood generated by admissible geodesic pairs, and if the other assumptions of Theorem~\ref{thm_main1} hold, then
$$
(z_0,\xi_0)\notin WF(\hat{f}(\lambda, \,\cdot\,)).
$$
Especially, if we $\hat{\xi_0}$ satisfies the strict Stefanov-Uhlmann regularity condition, we show the above to be true.

So far we have shown that multiplying Gaussian beam quasimodes, near a given point of interest $z_0$, and using~\ref{integral_O_infty}, produces an integral transformation of $\hat{f}(\lambda,x)$ in the $x$-variable satisfying
\begin{equation}\label{rapid_decay}
\int_{M_0} \hat{f}(\lambda,x) u_\tau(\xi,x) \,dV_{g_0} = O((\tau\abs{\xi})^{-\infty}),
\end{equation}
where $u_\tau$ is given in~\eqref{def_of_u_tau}, $\xi\in \mathcal{U}$ and $(M_0,g_0)$ is the transversal manifold.

We may define a function  
$$
k_\tau(\xi,x)=\abs{\xi}^{\frac{m}{4}}\tau^{\frac{m+2}{4}}u_\tau(\xi,x)
$$
on $\mathcal{U}\times U$, and by Proposition~\eqref{amplitude_prop} $k_\tau=e^{i \tau \Phi}B_\tau$, where $B_\tau$ is a polyhomogenous symbol of class $S_{\text{phg}}^{\frac{3m}{4},\frac{m}{4}}$ in the sense of~\cite{WZ}. By~\eqref{rapid_decay}, the integral of $k_\tau$ against $\hat{f}(\lambda,\cdot)$ is of order $\abs{\xi}^{\frac{m}{4}}O((\tau\abs{\xi})^{-\infty})$.

We are ready to prove our main result. The proof is a direct application of Theorem 4.8. of~\cite{WZ}.
\begin{proof}[Proof of Theorem~\ref{thm_main1}]
Let $(z_0,\xi_0)\in T^*M_0$ so that $\xi_0$ has a neighborhood in $T^*M_0$ generated by admissible geodesic pairs.  
Thus, the argument leading to~\eqref{integral_O_infty} combined with the construction in Section~\ref{sec_quasimodes} implies that
$$
\int_{M_0} \hat{f}(\lambda,x) k_\tau(\xi,x) \,dV_{g_0}(x) = O(\tau^{-\infty})
$$
for $\xi$ belonging to some bounded neighborhood of $\xi_0$ in $T^*M_0$. By Theorem 4.8. of~\cite{WZ} we have
$$
(z_0,\xi_0)\notin WF(\hat{f}(\lambda,\cdot)).
$$
\end{proof}

\section{APPENDIX}
\begin{proof}[Proof of Lemma~\ref{lemma_crossing_geodesics_parametrization}]
Let $y$ be normal coordinates near $z_0$, and write 
\[
\gamma_0 = a_j dy^j|_{z_0}
\]
where $(a_1, \ldots, a_n)$ is a unit vector in $\mR^n$. We then define the unit covectors 
\[
\gamma(z) := \frac{a_j dy^j|_z}{\langle a_j dy^j|_z, a_j dy^j|_z \rangle^{1/2}}
\]
and 
\begin{align*}
\omega_1(z,\omega) &:= \frac{\gamma(z) - \langle \gamma(z), \omega \rangle \omega + c(z,\omega) \omega}{[1 - \langle \gamma(z), \omega \rangle^2 +c(z,\omega)^2]^{1/2}}, \\
\omega_2(z,\omega) &:= 2 \langle \omega_1(z,\omega), \omega \rangle \omega - \omega_1(z,\omega)
\end{align*}
for $(z,\omega)$ near $(z_0,\omega_0)$, where $c(z,\omega)$ is chosen so that 
\[
\langle \omega_1(z,\omega), \omega \rangle = \langle \gamma_0, \omega_0 \rangle.
\]
A computation shows that the right choice for $c$ is 
\[
c(z,\omega) := \langle \gamma_0, \omega_0 \rangle \left[ \frac{1 - \langle \gamma(z), \omega \rangle^2}{1 - \langle \gamma_0, \omega_0 \rangle^2} \right]^{1/2}.
\]
Then $\omega_1$ and $\omega_2$ depend smoothly on $(z,\omega)$ near $(z_0,\omega_0)$, and they satisfy 
\begin{gather*}
\omega_1(z_0,\omega_0) = \gamma_0, \qquad \omega_2(z_0,\omega_0) = \tilde{\gamma}_0, \\
\omega_1(z,\omega) + \omega_2(z,\omega) = t_0 \omega
\end{gather*}
where $t_0 = 2 \langle \gamma_0, \omega_0 \rangle$ is a constant.
\end{proof}

Another way to do the above would be the following:
\begin{proof}[Alternative proof of Lemma~\ref{lemma_crossing_geodesics_parametrization}]
 Let $\xi_0\in S^*_{z_0}M$, and assume that $\zeta_i\in S^*M$, $i=1,2$, are such that 
 \begin{equation}\label{sum_generates}
 \zeta_1+\zeta_2=t_0\xi_0
 \end{equation} Let $U$ be a neighborhood of $z_0$ where the exponential map is defined. We define the parametrization $I=(\omega_1,\omega_2)$ as follows. We set for $\xi\in S^*U$
 \begin{equation*}\label{define_param}
 \omega_i(\xi)=P\circ \left (\norm{(P^{-1}\xi)^\parallel}O_{\xi}\zeta_i+\frac{1}{2}(P^{-1}\xi)^\bot\right).
 \end{equation*}
 Here $P$ stands for the parallel translation $S^*_{z_0}M\to S^*_{\pi(\xi)}M$ along unit speed geodesic with $\exp_{z_0}^{-1}(\pi(\xi))$ as initial data for unit time $t=1$, and $P^{-1}$ is its inverse $S^*_{\pi(\xi)}M\to S^*_{z_0}M$.
 Here we have orthogonally decomposed 
 $$
 P^{-1}\xi=(P^{-1}\xi)^\parallel+(P^{-1}\xi)^\bot,
 $$
 to the part $(P^{-1}\xi)^\parallel$ in the plane $V$ spanned by $\zeta_1$ and $\zeta_2$ and to the part orthogonal to $V$. Since $\zeta_1+\zeta_2=t_0\xi_0$, we have that $\xi_0\in V$. Above $O_{\xi}$ 
 is the unique rotation on the plane, an element of $SO(1)$, that takes $\xi_0$ to be parallel with $(P^{-1}\xi)^{\parallel}$. Thus $O_\xi$ satisfies
 \begin{equation}\label{def_of_O}
 \left[O_{\xi}\right]\xi_0=\frac{(P^{-1}\xi)^{\parallel}}{\norm{(P^{-1}\xi)^{\parallel}}}.
 \end{equation}

 We have 
 \begin{align*}
&\left(\norm{(P^{-1}\xi)^\parallel}O_{\xi}\zeta_1+\frac{t_0}{2}(P^{-1}\xi)^\bot\right)+\left(\norm{(P^{-1}\xi)^\parallel}O_{\xi}\zeta_2+\frac{t_0}{2}(P^{-1}\xi)^\bot\right) \\
&=t_0\left((P^{-1}\xi)^{\parallel}+(P^{-1}\xi)^\bot\right)=t_0(P^{-1}\xi),
 \end{align*}
 where we have used equations~\eqref{sum_generates} and~\eqref{def_of_O}. Consequently
 $$
 \omega_1(\xi)+\omega_2(\xi)=t_0\xi
 $$
 by linearity of the parallel translation. We also have $\omega_i(\xi_0)=\zeta_i$. The parametrization is well defined and smooth as long as $(P^{-1}\xi)^\parallel\neq 0$. Thus by continuity of the exponential map and of the parallel translation, the parametrization is well defined and smooth on some neighborhood of $\xi_0$.
\end{proof}

The following lemma gives a smooth local parametrization for orthonormal coframes $F\in FO_\xi(M)$ as a function of $\xi\in S^*M$ near a given $\xi_0\in S^*M$.
\begin{lemma}\label{frame_parametrization}
 Let $\xi_0\in S_{z_0}^*M$.
 Then there exists an open neighborhood $U_S$ of $\xi_0$ in $S^*M$ and a $C^\infty$ smooth mapping $F:U_S\to FO_\xi(M)$, 
 $$
 F(\xi)=(\omega^1(\xi),\ldots,\omega^n(\xi)),
 $$
 such that
\begin{equation}\label{frame_points_to}
\omega^1(\xi)=\xi
\end{equation}
and $\pi(F(\xi))=\pi(\xi)$. (This latter condition just means $\xi$ over a point $z\in M$ is mapped to a coframe $F(\xi)$ over the same point $z$.)
\end{lemma}
\begin{proof}
 Let $\xi_0\in S_{z_0}^*M$, let $F_0\in FO^*_{z_0}(M)$ be an orthonormal coframe $(\omega^1,\ldots, \omega^n)$ of $T^*_{z_0}M$ and let $U$ be a neighborhood of $z_0$ where the exponential map is defined. For $\xi\in S^*U$, we define
 $$
 F(\xi)=P\circ O_0\circ P^{-1}(\xi).
 $$
 Here $P$ is the parallel translation (either of a covector or a coframe) along a geodesic with initial data $\exp_{z_0}^{-1}(\pi(\xi))\in T_{z_0}^*M$ for a unit time.
 
 The mapping $O_0: S^*_{z_0}M\to FO^*_{z_0}(M)$ is a unique rotation of the fixed coframe $F_0$ defined as follows:
 Let $\omega\in S_{z_0}^*M$, and let $V_\omega$ be the plane spanned by $\omega$ and by the first covector $\omega^1$ of the  coframe $F_0$. We define $O_0(\omega)$ to be the rotation of the coframe $F_0$ so that $\omega^1$ is rotated to $\omega$ on the plane $V_\omega$ while directions initially orthogonal to $V_\omega$ remains orthogonal to $V_\omega$ under the rotation. 
 
 More precisely, we may split the cotangent space $T_{z_0}^*M$ as $V_\omega\oplus H_\omega$, where $\oplus$ stands for an orthogonal direct sum. Letting $(\omega^1,\omega)$ be a (not necessarily orthogonal) basis for $V_\omega$, and choosing some basis for $H_\omega$, we have a matrix representation for an element of $\mathcal{R}_\omega\in SO(n)$ as
 $$
 \left[\begin{array}{ccc}
	0 & -1 & \\
	1 & 0 &  \\
	 &  & I_H \\
	\end{array}\right].
 $$
 Now $\mathcal{R}_\omega$ induces a rotation of the frame $F_0$, which we define to be $O_0(\omega)$. The first component of this frame satisfies
 $$
 [O_0(\omega)]^1=\omega, \mbox{ i.e. } \omega^1\mapsto \omega.
 $$
 Since $V_\omega$ and its orthogonal complement $H_\omega$ depend smoothly on $\omega$ and $\mathcal{R}_\omega$ is independent of the basis of $H_\omega$, we have that $O_0$ is well defined and depends smoothly on $\omega$.
 
 All the steps in the composition defining $F(\xi)$ are smooth, and thus $F(\xi)$ depends smoothly on $\xi\in S^*U$. We define $U_S=S^*U$.
\end{proof}

\bibliographystyle{alpha}

\providecommand{\bysame}{\leavevmode\hbox to3em{\hrulefill}\thinspace}
\providecommand{\href}[2]{#2}

\end{document}